\documentclass[11pt]{article}
\usepackage{a4wide}

\usepackage[a4paper, margin=2.3cm]{geometry}
\usepackage{amsmath,amssymb,amsthm}
\usepackage{color}
\usepackage{parskip}
\usepackage{hyperref}
\usepackage{parskip}
\usepackage{setspace}
\usepackage{tikz}
\usepackage{graphicx}

\usepackage{varioref}
\usepackage[ddmmyyyy]{datetime}
\usepackage{mathtools}
\usepackage[h]{esvect}
\usepackage{array}
\usepackage{listings}
\usepackage{microtype}
\usepackage{booktabs}
\usepackage{braket} 
\usepackage{graphicx}
\usepackage{adjustbox}
\usepackage[ruled,vlined]{algorithm2e}
\usepackage{enumitem}
\usepackage[OT1]{fontenc}
\usepackage[english]{babel}
\usepackage[utf8]{inputenc}
\usepackage[sc]{mathpazo}

\newtheorem{theorem}{Theorem}[section]

\newtheorem{corollary}[theorem]{Corollary}
\newtheorem{proposition}[theorem]{Proposition}
\newtheorem{lemma}[theorem]{Lemma}

\newtheorem{definition}[theorem]{Definition}

\def\F{F}

\def\Fq{\mathbb{F}_q}
\def\R{\mathbb{R}}

\def\F{{\mathbb F}}
\def\R{{\mathbb R}}

\DeclareMathOperator{\rank}{rank}

\newcommand*\circled[1]{\tikz[baseline=(char.base)]{
            \node[shape=circle,draw,inner sep=1pt] (char) {#1};}}

\begin{document}
\pagenumbering{arabic}
\title{An energy decomposition theorem for matrices \\and related questions}
\author{
Ali Mohammadi\thanks{ School of Mathematics, Institute for Research in Fundamental Sciences (IPM), Tehran, Iran. Email:~\href{mailto:a.mohammadi@ipm.ir}{a.mohammadi@ipm.ir}}
\and 
    Thang Pham\thanks{Theory of Combinatorial Algorithms Group, ETH Zürich, Switzerland. Email: \href{mailto:phamanhthang.vnu@gmail.com}{phamanhthang.vnu@gmail.com}}
  \and
  Yiting Wang \thanks{Department of Computer Science, ETH Zürich, Switzerland. Email: \href{mailto:yitwang@student.ethz.ch}{yitwang@student.ethz.ch}}
}
\maketitle
\begin{abstract}
    Given $A\subseteq GL_2(\F_q)$, we prove that there exist disjoint subsets $B, C\subseteq A$ such that $A = B \sqcup C$ and their additive and multiplicative energies satisfying
\[
\max\{\,E_{+}(B),\, E_{\times}(C)\,\}\ll \frac{|A|^3}{M(|A|)},
\]
where
\begin{equation*}
\label{eqn:MAminBVPolyLSSS}
    M(|A|) = \min\Bigg\{\,\frac{q^{4/3}}{|A|^{1/3}(\log|A|)^{2/3}},\, \frac{|A|^{4/5}}{q^{13/5}(\log|A|)^{27/10}}\,\Bigg\}.
\end{equation*}
We also study some related questions on moderate expanders over matrix rings, namely, for $A, B, C\subseteq GL_2(\mathbb{F}_q)$, we have 
\[|AB+C|, ~|(A+B)C|\gg q^4,\]
whenever $|A||B||C|\gg q^{10 + 1/2}$. These improve earlier results due to Karabulut, Koh, Pham, Shen, and Vinh (2019). 
\end{abstract}
\section{Introduction}

Let $\F_q$ denote a finite field of order $q$ and characteristic $p$ and let $M_2(\F_q)$ be the set of two-by-two matrices with entries in $\F_q$. We write $X\ll Y$ or $X=O(Y)$ to mean $X\leq CY$ for some absolute constant $C>0$ and use $X\sim Y$ if $Y\ll X\ll Y$. 

Given subsets $A, B\subseteq M_2(\F_q)$, we define the sum set $A+B$ to be the set $\{a+b: (a, b)\in A\times B\}$ and similarly define the product set $AB$. In this paper, we study various questions closely related to the sum-product problem over $M_2(\F_q)$, which is to determine nontrivial lower bounds on the quantity $\max\{\,|A+A|,\, |AA|\,\}$, under natural conditions on sets $A\subseteq M_2(\F_q)$. 

A result in this direction was proved by Karabulut et al. in \cite[Theorem~1.12]{KaKoPhShVi}, showing that if $A\subseteq M_2(\F_q)$ satisfies $|A|\gg q^3$ then
\begin{equation}
\label{eqn:SRB2}
\max\{\,|A+A|,\, |AA|\,\}\gg \min\left\{\,\frac{|A|^2}{q^{7/2}},\, q^{2}|A|^{1/2}\,\right\}.
\end{equation}

A closely related quantity is the additive energy $E_+(A, B)$ defined as the number of solutions $(a, a^\prime, b, b^\prime)\in A^2\times B^2$ such that $a + b = a^\prime + b^\prime$. The multiplicative energy $E_\times(A, B)$ is defined in a similar manner. We also use, for example, $E_+(A) = E_+(A, A)$. For $\lambda\in M_2(\F_q)$, we define the representation function $r_{A B}(\lambda) = |\{\,(a, b)\in A\times B: a b = \lambda\,\}|$. Note that $r_{A B}$ is supported on the set $AB$ and so we have the identities
\begin{equation}
\label{eqn:1st&2ndM}
    \sum_{\lambda\in A B}r_{A B}(\lambda) = |A||B|\quad \text{and}\quad \sum_{\lambda\in A B}r_{A B}(\lambda)^2= E_{\times}(A, B).
\end{equation}
A standard application of the Cauchy-Schwarz inequality gives
\begin{equation}
\label{eqn:AddECS}
  E_+(A, B)|A+B|,\;  E_\times(A, B)|AB|\geq |A|^2|B|^2.
\end{equation}

Balog and Wooley~\cite{BW} initiated the investigation into a type of energy variant of the sum-product problem by proving that given a finite set $A\subset \R$, one may write $A= B \sqcup C$ such that $\max\{E_+(B), E_\times(C)\}\ll |A|^{3-\delta}$ for some explicit value $0<\delta<1$. The authors also proved a finite field variant of this result. Also, see \cite{HH, MohSte, RNShWin, RudShSt} for quantitative improvements, analogues and various applications of these results.

The main goal of this paper is to study energy variants of the sum-product problem over the ring $M_2(\F_q)$ and in particular to obtain a low energy decomposition result similar to those in \cite{BW}. It is important to note that $M_2(\F_q)$ is not a commutative ring and, in part due to this fact, results known over finite fields do not readily extend to this setting. 

\section{Main results}
Our first theorem is on an energy decomposition of a set of matrices in $M_2(\mathbb{F}_q)$.
\begin{theorem}
\label{thm:BWDM2}
Given $A\subseteq GL_2(\F_q)$, there exist disjoint subsets $B, C\subseteq A$ such that $A = B \sqcup C$ and
\[
\max\{E_{+}(B), E_{\times}(C)\}\ll \frac{|A|^3}{M(|A|)},
\]
where
\begin{equation}
\label{eqn:MAminBVPolyLS}
    M(|A|) = \min\Bigg\{\,\frac{q^{4/3}}{|A|^{1/3}(\log|A|)^{2/3}},\, \frac{|A|^{4/5}}{q^{13/5}(\log|A|)^{27/10}}\,\Bigg\}.
\end{equation}
\end{theorem}
It follows from this theorem that for any set $A$ of matrices in $M_2(\mathbb{F}_q)$, we always can find a subset with either small additive energy or small multiplicative energy. In the setting of finite fields, such a result has many applications in studying exponential sums and other topics, for instance, see \cite{MohSte, pham, RNShWin, Shk16, Shk20, SwaWin, VA} and references therein. By the Cauchy-Schwarz inequality, we have the following direct consequence  on a sum-product estimate, namely, for $A\subseteq GL_2(\mathbb{F}_q)$, we have 
\begin{equation}\label{eqn:EDSP}
\max \left\lbrace |A+A|, |AA| \right\rbrace \gg |A|\cdot M(|A|).
\end{equation}
By a direct computation, one can check that this is better than the estimate (\ref{eqn:SRB2}) in the range $|A|\ll q^{3+5/8}/(\log |A|)^{1/2}$. 

In the next theorem, we show that the lower bound of \eqref{eqn:EDSP} can be improved by a direct energy estimate.
\begin{theorem}
\label{thm:AddEng}
Let $A, B\subseteq M_2(\F_q)$ and $C\subseteq GL_2(\F_q)$. Then
\[
E_+(A, B) \ll \frac{|A|^2|BC|^2}{q^4} + q^{13/2}\frac{|A||BC|}{|C|}.
\]
\end{theorem}

\begin{corollary}\label{sum-productestiamte}
For $A\subseteq M_2(\F_q)$, with $|A|\gg q^{3}$, we have
\begin{equation}
\label{eqn:SRB1}
\max\{\,|A+A|,\, |AA|\,\}\gg \min\left\{\,\frac{|A|^2}{q^{13/4}},\, q^{4/3}|A|^{2/3}\,\right\}.
\end{equation}
In addition, if $|AA|\ll |A|$ and $|A|\gg q^{3+\frac{1}{2}}$, then \begin{equation}\label{eqq99}|A+A|\gg q^4.\end{equation}
\end{corollary}

We point out that the estimate \eqref{eqn:SRB1} improves \eqref{eqn:SRB2} in the range $|A|\ll q^{3+5/8}$ and is strictly stronger than \eqref{eqn:EDSP}. We also note that our assumption to get the estimate (\ref{eqq99}) is reasonable. For instance, let $G$ be a subgroup of $\mathbb{F}_q^*$, and $A$ be the set of matrices with determinants in $G$, then we have $|A|\sim q^3\cdot |G|$ and $|AA|=|A|$. 

It has been proved in \cite[Theorems 1.8 and 1.9]{KaKoPhShVi} that for $A, B, C\subseteq M_2(\mathbb{F}_q)$, if $|A||B||C|\ge q^{11}$, then we have 
\[|AB+C|, ~|(A+B)C|\gg q^4.\]
In the following theorem, we provide improvements of these results. 
\begin{theorem}
\label{thm:3vexp}
Let $A, B, C \subseteq M_2(\F_q)$, we have 
\[|AB+C|\gg \min \left\lbrace\, q^4,\, \frac{|A||B||C|}{q^{13/2}} \,\right\rbrace.\]
If $C\subseteq GL_2(\mathbb{F}_q)$, the same conclusion holds for $(A+B)C$, i.e. 
\[|(A+B)C|\gg \min \left\lbrace\, q^4,\, \frac{|A||B||C|}{q^{13/2}} \,\right\rbrace.\]
In particular, 
\begin{enumerate}
\item if  $|A||B||C|\gg q^{10 + 1/2}$, then 
$|AB + C|\gg q^4.$
\item if $|A||B||C|\gg q^{10 + 1/2}$ and $C\subseteq GL_2(\mathbb{F}_q)$, then 
$|(A+B)C|\gg q^4.$
\end{enumerate}
\end{theorem}

The condition $C\subseteq GL_2(\mathbb{F}_q)$ is necessary, since, for instance, one can take $C$ to be the set of matrices with zero determinant and $A=B=M_2(\mathbb{F}_q)$, then $|(A+B)C|\sim q^3$ and $|A||B||C|\sim q^{11}$. 

We expect that the exponent $q^{10 + 1/2}$, in the final conclusions of the above theorem, could be further improved to $q^{10}$, which, as we shall demonstrate, is sharp. For $AB+C$, let $A$ and $B$ be the set of lower triangular matrices in $M_2(\F_q)$ and for arbitrary $0<\delta<1$, let $X\subseteq \F_q$ be any set with $|X|= q^{1-\delta}$ and let
\[
C = \left\{\,\begin{pmatrix} c_1 & c_2 \\ c_3 & c_4 \end{pmatrix}:c_1, c_3, c_4\in \F_q, c_2\in X\,\right\}.
\]
Then $|A||B||C| = q^{10-\delta}$ and $|AB+C| = |C| = q^{4-\delta}$. 

For $(A+B)C$, the construction is as follows:
For arbitrary $k$, let $q=p^k$ and let $V$ denote a $(k-1)$-dimensional vector space over $\F_p$ in $\F_q$. Thus, we have $|V| = p^{k-1} = q^{1-1/k}$. Now let 
\[
A = B = \left\{\,\begin{pmatrix} x_1 & x_2 \\ x_3 & x_4 \end{pmatrix}:x_1, x_2\in V, x_3, x_4\in \F_q\,\right\}
\]
and
\[
C = \left\{\,\begin{pmatrix} c_1 & c_2 \\ c_3 & c_4 \end{pmatrix}:c_1, c_3\in \F_q, c_2, c_4\in \F_p\,\right\}.
\]
Note that
\[
(A+B)C = AC =  \left\{\,\begin{pmatrix} y_1 & y_2 \\ y_3 & y_4 \end{pmatrix}:y_1, y_3, y_4\in \F_q, y_2\in V\,\right\},
\]
where we have used that $V\cdot \F_p + V\cdot \F_p = V+V = V.$

Thus, $|A||B||C| = (q^2\cdot q^{2-2/k})^2 \cdot (q^2\cdot q^{2/k}) = q^{10-2/k}$ while $|(A+B)C| = q^{4-1/k}$.



It is worth noting that in the setting of finite fields, our approach and that of Karabulut et al. in \cite{KaKoPhShVi} imply the same result. Namely, for $A,B,C \subseteq \Fq$, we have $|(A+B)C|, |AB+C| \gg q$ whenever $|A||B||C| \gg q^2$. Thus, there exists a different phenomenon between problems over finite fields and over the matrix ring $M_2(\Fq)$.

Let $A, B, C, D\subseteq M_2(\F_q)$, our last theorem is devoted to the solvability of the equation 
\[x+y=zt, ~x\in A, y\in B, z\in C, t\in D.\]
Let $\mathcal{J}$ denote the number of solutions to this equation. 

One can check that by using Theorem 4.2 and Lemma 4.1 from \cite{KaKoPhShVi}, one has 
\begin{equation}\label{eqcompare}\left\vert \mathcal{J}-\frac{|A||B||C||D|}{q^4} \right\vert\ll q^{7/2}(|A||B||C||D|)^{1/2}.\end{equation}
Thus, when $|A||B||C||D|\gg q^{15}$, then $\mathcal{J}\sim \frac{|A||B||C||D|}{q^4}$. We refer the interested reader to \cite{Sar} for such a result over finite fields. In our last theorem, we are interested in bounding $\mathcal{J}$ from above when $|A||B||C||D|$ is smaller. 
\begin{theorem}
\label{thm:SarkM2v2}
Let $A, B, C, D\subseteq M_2(\F_q)$ and let $\mathcal{J}$ denote the number of solutions to the equation
\begin{equation*}
a+b = c d, \quad (a, b, c, d)\in A\times B\times C\times D.
\end{equation*}
Then, we have
\begin{equation*}
    \mathcal{J} \ll  \frac{|A||B|^{1/2}|C||D|}{q^2} + q^{13/4} (|A||B||C||D|)^{1/2}.
\end{equation*}
\end{theorem}
Assuming $|A|=|B|=|C|=|D|$, the upper bound of this theorem is stronger than that of (\ref{eqcompare}) when $|A|\ll q^{11/3}$. 
\paragraph{Organization.} The rest of this paper is structured as follows: In the next section, we prove a preliminary lemma, which is one of the key ingredients in the proof of our energy decomposition theorem. Section $4$ is devoted to the proof of Theorem \ref{thm:BWDM2}. The proofs of Theorem~\ref{thm:AddEng} and Corollary~\ref{sum-productestiamte} will be presented in Section $5$. Section $6$ contains proofs of Theorem~\ref{thm:3vexp}, and Theorem \ref{thm:SarkM2v2}.
\section{A preliminary lemma}
Given sets $A, B, C, D, E, F\subseteq M_2(\F_q)$, let $\mathcal{I}(A, B, C, D, E, F)$ be the number of solutions
\[
(a, b, c, d, e, f)\in A\times B\times C\times D\times E\times F: \quad ab+ ef = c + d.
\]
The main purpose of this section is to prove an estimate for $\mathcal{I}(A, B, C, D, E, F)$, which is one of the key ingredients in the proof of Theorem \ref{thm:BWDM2}.
\begin{proposition}
\label{thm:PPIM2}
We have
\[
\bigg|\mathcal{I}(A, B, C, D, E, F) - \frac{|A||B||C||D||E||F|}{q^4}\bigg| \ll  q^{13/2}\sqrt{|A||B||C||D||E||F|}\,.
\]
\end{proposition}

To prove Proposition~\ref{thm:PPIM2}, we define the sum-product digraph $G=(V, E)$ with the vertex set $V=M_2(\mathbb{F}_q)\times M_2(\mathbb{F}_q)\times M_2(\mathbb{F}_q)$, such that there is a directed edge going from $(a, e, c)$ to $(b, f, d)$ if and only if $ab+ ef = c+d$. The setting of this digraph is a generalization of that in \cite[Section~4.1]{KaKoPhShVi}

Let $G$ be a digraph on $n$ vertices. Suppose that $G$ is regular of degree $d$, i.e. the in-degree and out-degree of each vertex are equal to $d$.

Let $m_G$ be the adjacency matrix of $G$, where $M_{ij}=1$ if and only if there is a directed edge from $i$ to $j$. Let $\mu_1=d, \mu_2, \ldots, \mu_n$ be the eigenvalues of $m_G$. Notice that these eigenvalues can be complex numbers, and for all $2\le i\le n$, we have $|\mu_i|\le d$. Define $\mu(G):=\max_{|\mu_i|\ne d}|\mu_i|$. This value is referred to as the second largest eigenvalue of $m_G$.  

A digraph $G$ is called an $(n, d, \mu)$-digraph if $G$ is a $d$-regular graph of $n$ vertices, and the second largest eigenvalue is at most $\mu$.

We recall the following lemma from \cite{vu} on the distribution of edges between two vertex sets on an $(n, d, \mu)$-digraph.

\begin{lemma}\label{edge}
  Let $G = (V, E)$ be an $(n, d, \mu)$-digraph. For any two sets $B, C
  \subseteq V$, the number of directed edges from $B$ to $C$, denoted by $e(B, C)$ satisfies
  \[ \left| e(B, C) - \frac{d}{n}|B | |C| \right| \leq \mu \sqrt{|B| |C|}\,.\]
\end{lemma}

With Lemma \ref{edge} in hand, to prove Proposition \ref{thm:PPIM2}, it is enough to study properties of the sum-product digraph $G$. 

\begin{definition}
	Let $a, b \in M_2(\mathbb{F}_q)$. We say they are equivalent, if whenever the $i$-th row of $a$ is not all-zero, then neither is the $i$-th row of $b$ and vice versa, for $1\leq i \leq 2$.
\end{definition}

\begin{proposition}\label{thm: sum-product-graph}
	The sum product graph G is a $(q^{12},\, q^8,\, c\cdot q^{13/2})$-digraph, for some positive constant $c$.
\end{proposition}
\begin{proof}
The number of vertices is $|M_2(\mathbb{F}_q)|^3 = q^{12}$. Moreover, for each choice of $(b, f)$, $d$ is determined uniquely from $d = ab + ef - c$. Thus, there are $|M_2(\mathbb{F}_q)|^2 = q^8$ directed edges going out of each vertex. The number of incoming directed edges can be argued in the same way. To conclude, the digraph $G$ is $q^8$-regular. Let $m_G$ denotes the adjacency matrix of $G$. It remains to bound the magnitude of the second largest eigenvalue of the adjacency matrix of $G$, i.e., $\mu(m_G)$.

In the next step, we are going to show that $m_G$ is a normal matrix, i.e. $m_G^Tm_G=m_Gm_G^T$, where $m_G^T$ is the transpose of $m_G$. For a normal matrix $m$, we know that if $\lambda$ is an eigenvalue of $m$, then $|\lambda|^2$ is an eigenvalue of $mm^T$ and $m^Tm$. This comes from the fact that $\overline{\lambda}$ is an eigenvalue of $m^T$. Thus, for a normal matrix $m$, it is enough to give an upper bound for the second largest eigenvalue of $mm^T$ or $m^Tm$. 

There is a simple way to check whenever $G$ is normal or not. For any two vertices $u$ and $v$, let $\mathcal{N}^+(u,v)$ be the set of vertices $w$ such that $\overrightarrow{uw}, \overrightarrow{vw}$ are directed edges, and $\mathcal{N}^-(u, v)$ be the set of vertices $w'$ such that $\overrightarrow{w'u}, \overrightarrow{w'v}$ are directed edges. It is not hard to check that $m_G$ is normal if and only if $|\mathcal{N}^+(u,v)| = |\mathcal{N}^-(u,v)|$ for any two vertices $u$ and $v$. 

Given two vertices $(a, e, c)$ and $(a^\prime, e^\prime, c^\prime)$, where $(a, e, c) \neq (a^\prime, e^\prime, c^\prime)$, the number of $(x,y,z)$ that lies in the common outgoing neighborhood of both vertices is characterized by
	\begin{equation*}
		\begin{rcases}
		a x + e y = c + z\\
		a^\prime x + e^\prime y = c^\prime + z
		\end{rcases}
		\Longleftrightarrow (a - a^\prime)  x + (e - e^\prime)  y = (c - c^\prime)\,.
	\end{equation*}
For each pair $(x, y)$ satisfying this equation, $z$ is determined uniquely. Thus, the problem is reduced to computing the number of such pairs $(x, y)$. 

	For convenience, let $\bar{a} = a - a^\prime$, $\bar{c} = c - c^\prime$ and $\bar{e} = e - e^\prime$. Also, let $t = \begin{pmatrix} \bar{a} & \bar{e}\end{pmatrix}_{2\times 4}$. Then, the above relation is equivalent to 
	\begin{equation} \label{eq: base equation}
		\begin{pmatrix} \bar{a} & \bar{e}\end{pmatrix} \begin{pmatrix} x \\ y\end{pmatrix} =  t  \begin{pmatrix} x \\ y\end{pmatrix}_{4\times 2}= \bar{c}\,.
	\end{equation}
	 We now have the following cases: 
	\begin{itemize}
	\item (\textbf{Case 1}: $\rank(t) = 0$) Note that in this case, we need $a = a^\prime$, $c = c^\prime$ and $e = e^\prime$, which contradicts our assumption that $(a, e, c) \neq (a^\prime, e^\prime, c^\prime)$. Thus, we simply exclude this case.
	\item (\textbf{Case 2}: $\rank(t) = 1$) As $t$ is not an all-zero matrix, there is at least one non-zero row. Without loss of generality, assume it is the first row. Then, $t = \begin{pmatrix} a_1 & a_2 & e_1 & e_2 \\ \alpha a_1 & \alpha a_2 & \alpha e_1 & \alpha e_2\end{pmatrix}$, where $(a_1, a_2, e_1, e_2) \neq \textbf{0}$ and $\alpha \in \mathbb{F}_q$. 
		\begin{itemize}
			\item (\textbf{Case 2.1}: $\rank(\bar{c}) = 2$) In this case, there is no solution, as $\rank\left(t  \begin{pmatrix} x \\ y\end{pmatrix}\right) \leq \rank(t) = 1$ but $\rank(\bar{c}) = 2$.
			\item (\textbf{Case 2.2}: $\rank(\bar{c}) = 1$) Let $x = \begin{pmatrix} x_1 & x_2 \\ x_3 & x_4 \end{pmatrix}$, $y = \begin{pmatrix} y_1 & y_2 \\ y_3 & y_4 \end{pmatrix}$. We discuss two sub-cases: 	
	
	(a) $\bar{c} = \begin{pmatrix} c_1 & c_2 \\ \alpha c_1 & \alpha c_2 \end{pmatrix}$ with the same factor $\alpha$, where $(c_1, c_2) \neq (0, 0)$. In this case, we have the following set of equations:
				\begin{equation*}
					\begin{cases}
					a_1 x_1 + a_2 x_3 + e_1 y_1 + e_2 y_3 = c_1 \\
					a_1 x_2 + a_2 x_4 + e_1 y_2 	+ e_2 y_4 = c_2
				\end{cases}
				\,.
				\end{equation*}
			Since we assume $(a_1, a_2, e_1, e_2) \neq \textbf{0}$, without loss of generality, let $a_1 \neq 0$. Then, 
			\begin{equation*}
				\begin{cases}
				x_1 = (a_1)^{-1}(c_1 -a_2 x_3 -e_1 y_1 - e_2 y_3) \\
				x_2 = (a_1)^{-1}(c_2 - a_2 x_4 - e_1 y_2 - e_2 y_4)
			\end{cases}
			,
			\end{equation*}
			which means that for each $(x_3, y_1, y_3)$ there is a unique $x_1$ and for each $(x_4, y_2, y_4)$ there is a unique $x_2$. Thus, there are $q^6$ different $(x,y,z)$ solutions.
			
			(b) In all other sub-cases, there is no solution. If $\bar{c} = \begin{pmatrix} c_1 & c_2 \\ \beta c_1 & \beta c_2 \end{pmatrix}$, where $\beta \neq \alpha$ and $(c_1, c_2) \neq (0, 0)$, then we get the following two equations:
			\begin{equation*}
					\begin{cases}
					a_1 x_1 + a_2 x_3 + e_1 y_1 + e_2 y_3 = c_1 \\
					\alpha a_1 x_1 + \alpha a_2 x_3 + \alpha e_1 y_1  + \alpha e_2 y_4 = \beta c_1
				\end{cases}\,,
			\end{equation*}
			which obviously do not have any solution.
			
			Otherwise, $\bar{c} = \begin{pmatrix} \beta c_1 & \beta c_2 \\  c_1 & c_2 \end{pmatrix}$ where $(c_1, c_2) \neq (0, 0)$. Note that  if $\alpha \neq 0$, then $\beta \neq \alpha^{-1}$, because this case is covered in Case 2.2(a) implicitly. We get the following equations.
			\begin{equation*}
					\begin{cases}
					a_1 x_1 + a_2 x_3 + e_1 y_1 + e_2 y_3 = \beta c_1 \\
					\alpha a_1 x_1 + \alpha a_2 x_3 + \alpha e_1 y_1 + \alpha e_2 y_3 = c_1
				\end{cases}\,,
			\end{equation*}
			which obviously do not have any solution. Notice that $\alpha = 0$ or $\beta = 0$ corresponds to $t$ and $\bar{c}$ not being equivalent.
			
			\item (\textbf{Case 2.3}: $\rank(\bar{c}) = 0$) This case is similar to the Case 2.2(a), except $c_1 = c_2 = 0$. We have the following two equations:
			\begin{equation*}
				\begin{cases}
					a_1 x_1 + a_2 x_3 + e_1 y_1 + e_2 y_3 = 0 \\
					a_1 x_2 + a_2 x_4 + e_1 y_2 	+ e_2 y_4 =0
				\end{cases}\,.
			\end{equation*}
			Following the same analysis, we conclude there are $q^6$ solutions.
		\end{itemize}
	\item (\textbf{Case 3}: $\rank(t) = 2$) In this case, we always have solutions, for any $\bar{c}$. 
	\begin{itemize}
		\item (\textbf{Case 3.1}: $\rank(\bar{a}) = 2$ or $\rank(\bar{e}) = 2$) In this case, let us look back on equation (\ref{eq: base equation}). If $\rank(\bar{a}) = 2$, then we can rewrite (\ref{eq: base equation}) as $\bar{a}x = \bar{c} - \bar{e} y$. Observe that, for any $y\in M_2(\mathbb{F}_q)$, there is a unique $x$. Thus, the number of solutions is $q^4$. The case where $\rank(\bar{e}) = 2$ is similar.
		\item (\textbf{Case 3.2}: $\rank(\bar{a}) \leq 1$ and $\rank(\bar{e}) \leq 1$) In this case, it is not hard to observe that $T$ must be one of the following four types:
		\begin{enumerate}[label=(\roman*)]
			\setlength\itemsep{1em}
			\item $\begin{pmatrix} a_1 & a_2 & e_1 & e_2 \\ \alpha a_1 & \alpha a_2 &\beta e_1 &\beta e_2\end{pmatrix}$, where $(a_1, a_2), (e_1, e_2)\neq (0,0)$, $\alpha \neq \beta$, $(\alpha, \beta) \neq (0,0)$.
			\item $\begin{pmatrix} \alpha a_1 & \alpha a_2 &\beta e_1 &\beta e_2 \\ a_1 & a_2 & e_1 & e_2  \end{pmatrix}$, where $(a_1, a_2), (e_1, e_2) \neq (0,0)$, $\alpha \neq \beta$, $(\alpha, \beta) \neq (0,0)$.
			\item $\begin{pmatrix}  a_1 & a_2 &0 &0 \\ 0& 0 & e_1 & e_2  \end{pmatrix}$, where $(a_1, a_2),(e_1, e_2) \neq (0,0)$.
			\item $\begin{pmatrix} 0 & 0 & e_1 & e_2 \\ a_1& a_2 & 0 & 0  \end{pmatrix}$, where $(a_1, a_2), (e_1, e_2)\neq (0,0)$.
		\end{enumerate}
		Since $(i)$ and $(ii)$ are symmetric and so is $(iii)$ and $(iv)$, we only argue for $(i)$ and $(iii)$.
		For $(iii)$, reusing notations from Case 2.2(a), we have
		\begin{equation*}
			\begin{cases}
			a_1 x_1 + a_2 x_3 = c_1 \\
			a_1 x_2 + a_2 x_4 = c_2 \\
			e_1 y_1 + e_2 y_3 = c_3 \\
			e_1 y_2 + e_2 y_4 = c_4
		\end{cases}
		.
		\end{equation*}
		
		As $(a_1, a_2) \neq (0,0)$ and $(e_1, e_2) \neq (0,0)$, without loss of generality, we assume $a_1 \neq 0$ and $e_1 \neq 0$. Then, it means for each $(x_3, x_4, y_3, y_4)$ there is a unique $(x_1, x_2, y_1, y_2)$. Thus, the system has $q^4$ solutions.
		
		For $(i)$, we have 
		\begin{equation*}
			\begin{cases}
			a_1 x_1 + a_2 x_3 + e_1 y_1 + e_2 y_3 = c_1 \quad\circled{1}\\
			a_1 x_2 + a_2 x_4 + e_1 y_2 + e_2 y_4 = c_2 \quad\circled{2}\\
			\alpha a_1 x_1 + \alpha a_2 x_3 + \beta e_1 y_1 + \beta e_2 y_3 = c_3 \quad\circled{3}\\
			\alpha a_1 x_2 + \alpha a_2 x_4 + \beta e_1 y_2 + \beta e_2 y_4 = c_4 \quad\circled{4}
		\end{cases}
		.
		\end{equation*}
		Again, assume $a_1 \neq 0$ and $e_1 \neq 0$. Now, take $\circled{1}\times \alpha - \circled{3}$, we get $(\alpha - \beta)(e_1 y_1 + e_2 y_3) = \alpha c_1 - c_3 $. As $\alpha \neq \beta$, this means $e_1 y_1 + e_2 y_3 = (\alpha - \beta)^{-1}(\alpha c_1 - c_3)$. Thus, for each $y_3$, there is a unique $y_1$. Similarly, compute $\circled{1}\times \beta - \circled{3}$, and we get $ a_1 x_1 + a_2 x_3 = (\beta - \alpha)^{-1}(\beta c_1 - c_3)$, which means that for each $x_3$, we get a unique $x_1$. We can do the same for $\circled{2}$ and $\circled{4}$ and conclude that there are $q^4$ solutions.
	\end{itemize}
 	\end{itemize}

	Observe that all cases are disjoint and they together enumerate all possible relations between vertices $(a, e, c)$ and $(a^\prime, e^\prime, c^\prime)$. We computed $\mathcal{N}^+((a, e, c), (a^\prime, e^\prime, c^\prime))$ at above and the computation for $\mathcal{N}^-((a, e, c), (a^\prime, e^\prime, c^\prime))$ is similar. Thus, we know $m_G$ is normal. Note that each entry of $m_Gm_G^T$ can be interpreted as counting the number of common outgoing neighbors between two vertices. We can write $m_Gm_G^T$ as
	\begin{equation*}
	\begin{split}
		m_Gm_G^T &= (q^8 - q^4)I + q^4 J - q^4 E_{21} + (q^6 - q^4)E_{22a}\\
		&\qquad\qquad  -q^4 E_{22b} + (q^6 - q^4) E_{23} + (q^4 - q^4) E_{31} + (q^4 - q^4) E_{32} \\
		     &= (q^8 - q^4)I + q^4 J - q^4 E_{21} + (q^6 - q^4)E_{22a}- q^4 E_{22b} + (q^6 - q^4) E_{23}\,,
	\end{split}
	\end{equation*}

	where $I$ is the identity matrix, $J$ is the all one matrix and $E_{ij}$s are adjacency matrices, specifying which entries are involved. For example, for Case 2.3, all pairs $(a,e,c), (a^\prime, e^\prime, c^\prime)$ with $c = c^\prime$ and $\rank(t) = 1$ are involved. Thus, $E_{23}$ is an adjacency matrix of size $q^{12}\times q^{12}$ (containing all pairs $(a,e,c)$), with pairs of vertices satisfying this property marked 1 and all others marked 0.   
	
	Finally, observe that each subgraph  defined by the corresponding adjacency matrix $E_{ij}$ is regular. This is due to the fact that the condition does not depend on specific value of $(a,e,c)$. Starting from any vertex $(a,e,c)$, we can get to all possible $\bar{a}, \bar{e},\bar{c}$ by subtracting the correct $(a^\prime, e^\prime, c^\prime)$. Thus, for each case, we get the same number of $(a^\prime, e^\prime, c^\prime)$ that satisfies the condition. 
	
	Let $\kappa_{ij}$ denotes an upper bound on the magnitude of the largest eigenvalue of $E_{ij}$.
	Then, it is easy to see that $\kappa_{21} \ll q^9$, $\kappa_{22a} \ll q^7$, $\kappa_{22b} \ll q^8$ and $\kappa_{23}\ll q^5$. 
	For example, in Case 2.1, we have $\rank(t) = 1$ and $\rank(\bar{c}) = 2$. For a fixed $(a, e,c)$, the former implies that there are $O(q^5)$ possibilities for $a^\prime$ and $e^\prime$ while the latter implies there are $O(q^4)$ possibilities for $c^\prime$. Altogether, there are $O(q^9)$ possibilities for $(a^\prime, e^\prime, c^\prime)$ in Case 2.1. Because the graph induced by $E_{21}$ is regular, we have $\kappa_{21}\ll q^9$. Other cases can be deduced accordingly.
	
	The rest follows from a routine computation: let $v_2$ be an eigenvector corresponding to $\mu(G)$. Then, because $G$ is regular and connected (easy to see, there is no isolated vertex), $v_2$ is orthogonal to the all $1$ vector, which means $J\cdot v_2 = \mathbf{0}$. We now have
	\begin{equation*}
	\begin{split}
		\mu(m_G)^2 v_2 &= m_Gm_G^T \cdot v_2 = (q^8-q^4)I\cdot v_2 + (- q^4 E_{21} \\
		                   &\qquad + (q^6 - q^4)E_{22a} - q^4 E_{22b} + (q^6 - q^4) E_{23}) \cdot v_2 \ll q^{13} \cdot v_2 \,.
	\end{split}
	\end{equation*}
	Thus, $\mu(G)\ll q^{13/2}$.
\end{proof}
\begin{proof}[Proof of Proposition \ref{thm:PPIM2}]
It follows directly from Proposition~\ref{thm: sum-product-graph} and Lemma \ref{edge} that
\begin{equation*}
	\left\vert \mathcal{I}(A, B, C, D, E, F)- \frac{1}{q^4}|A||B||C||D||E||F| \right\vert \ll  q^{13/2} \sqrt{|A||B||C||D||E||F|}\,.
\end{equation*}
This completes the proof. 
\end{proof}

\section{Proof of Theorem \ref{thm:BWDM2}}
To prove Theorem \ref{thm:BWDM2}, we will also need several technical results. A proof of the following inequality may be found in \cite[Lemma~2.4]{RNShWin}.
\begin{lemma}
\label{lem:ESubadd}
Let $V_1, \dots, V_k$ be subsets of an Abelian group. Then
\begin{equation*}
\label{eqn:Esubadd1}
    E_{+}\bigg(\bigsqcup_{i=1}^{k}V_i\bigg) \leq  \bigg(\sum_{i=1}^k E_{+}(V_i)^{1/4}\bigg)^4.
\end{equation*}
\end{lemma}

The following lemma is taken from \cite{MohSte} and may also be extracted from \cite{RNShWin} and \cite{RudShSt}. Lemma~\ref{lem:EnergyEnergyPigeonh} is slightly different to its analogues over commutative rings as highlighted by the duality of the inequalities \eqref{eqn:DARepLB} and \eqref{eqn:DARepLB2}.
\begin{lemma}
\label{lem:EnergyEnergyPigeonh}
Let $X \subseteq GL_2(\mathbb{F}_q)$. There exist sets $X_*\subset X$, $D \subset XX$, as well as numbers $\tau$ and $\kappa$ satisfying
\begin{equation}
\label{eqn:tauUBLB}
    \frac{E_{\times}(X)}{2|X|^2}\leq \tau \leq |X|,
\end{equation}
\begin{equation}
\label{eqn:DUBLB}
    \frac{E_{\times}(X)}{\tau^2\cdot \log |X|}\ll |D| \ll (\log|X|)^6 \frac{|X_*|^4}{E_{\times}(X)},
\end{equation}
\begin{equation}
    \label{eqn:A*LNEnergy}
    |X_*|^2 \gg \frac{E_{\times}(X)}{|X| (\log|X|)^{7/2}},
\end{equation}
\begin{equation}
\label{eqn:A*KappaLB}
\kappa \gg \frac{|D|\tau}{|X_*|(\log|X|)^2}
\end{equation}
such that either
\begin{equation}
\label{eqn:DARepLB}
    r_{DX^{-1}}(x) \geq \kappa \quad \text{for all} \quad x\in X_*.
\end{equation}
or 
\begin{equation}
\label{eqn:DARepLB2}
    r_{X^{-1}D}(x) \geq \kappa \quad \text{for all} \quad x\in X_*.
\end{equation}
\end{lemma}

We need a dyadic pigeonhole argument, which can be found in \cite[Lemma 18]{MurPetri18}.
\begin{lemma}
\label{lem:dyadicPigeonhold}
	For $\Omega \subseteq M_2(\mathbb{F}_q)$, let $w,f: \Omega \rightarrow \R^+$ with $f(x) \leq M, \;\forall x\in \Omega$. Let $W = \sum_{x\in \Omega} w(x)$. If $\sum_{x\in \Omega} f(x) w(x) \geq K,$ then there exists a subset $D \subset \Omega$ and a number $\tau$ such that $\tau \leq f(x) < 2\tau$ for all $x\in D$ and	$K/(2W) \leq \tau \leq M\,.$ Moreover
	\begin{equation*}
		\frac{K}{2 + 2\log_2M} \leq \sum_{x\in D} f(x)w(x) \leq 2\tau \sum_{x\in D}w(x) \leq \min\{2\tau W,\, 4\tau^2 |D|\}\,.
	\end{equation*}
	\end{lemma}

\begin{proof}[Proof of Lemma \ref{lem:EnergyEnergyPigeonh}]
We use the identities in \eqref{eqn:1st&2ndM} and apply Lemma \ref{lem:dyadicPigeonhold}, by taking $\Omega = X X, \, f = w = r_{XX},\, M = |X|,\, K = E_\times(X)$ and $W = |X|^2$, to find a set $D \subset XX$ and a number $\tau$, satisfying \eqref{eqn:tauUBLB}, such that $D = \{\, \lambda \in XX: \tau \leq r_{XX}(\lambda) < 2\tau\,\}\,$ and 
\begin{equation}\label{eqn:t2dlb}
    \tau^2 |D| \gg E_\times(X) / \log |X|\,.
    \end{equation}
	
	Define $P_1 = \{\,(x,y)\in X\times X: xy \in D\,\}$ and $A_x = \{\,y: (x,y)\in P_1\,\}$ for $x \in X$. By the definition of $D$, we know that $\tau |D| \leq |P_1| < 2\tau |D|$. We can use Lemma \ref{lem:dyadicPigeonhold} again with $\Omega = X, f(x) = |A_x|, w = 1, M = W = |X|$ and $K = |P_1|$ to find a set $V\subset X$ and a number $\kappa_1$ such that $V = \{\,x\in X: \kappa_1 \leq |A_x| < 2\kappa_1\,\}\,$ and 
	\begin{equation}
	\label{eqn:kappa1}
		|V|\kappa_1 \gg |P_1|/\log|X| \gg \tau |D| / \log |X|\,.
	\end{equation}
	Now we split the analysis into two cases based on $|V|$:
	
	\textbf{Case 1} ($|V| \geq \kappa_1 (\log |X|)^{-1/2}$): In this case, we simply set $X_* = V$ and $\kappa = \kappa_1$. For each $x\in V$, there are at least $\kappa_1$ different $y$ such that $x y \in D$. Therefore, $r_{DX^{-1}}(x) \geq \kappa \;\forall x \in X_*$.
		
	\textbf{Case 2} ($|V| < \kappa_1 (\log |X|)^{-1/2}$): In this case, we find another pair $U, \kappa_2$ that satisfies $|U| \gg \kappa_2 (\log |X|)^{-1/2}$ and set $X_* = U$ and $\kappa = \kappa_2$. Let $P_2 = \{\,(x,y)\in P_1: x\in V\,\}$ and $B_y = \{\,x: (x,y)\in P_2\,\}$. By definition, we have $|P_2| \geq |V|\kappa_1$. We apply Lemma \ref{lem:dyadicPigeonhold} again, with $\Omega = X, f(y) = |B_y|, w = 1, K = |P_2|, W = M =|X|$ to get $U \subset X$ and a number $\kappa_2$ such that	$U = \{\,y\in X: \kappa_2 \leq |B_y| < 2\kappa_2\,\}\,$ and 
	\begin{equation}
	\label{eqn:kappa2}
		|U|\kappa_2 \gg |P_2|/\log|X| \geq \kappa_1 |V| /\log|X| \,.
	\end{equation}
	Combining this inequality with the assumption of this case ($\kappa_1 \geq |V|(\log |X|)^{1/2}$) and $|V| \geq \kappa_2$, we conclude $|U| \gg \kappa_2 (\log |X|)^{-1/2}$. We can then argue similarly as in Case 1 to conclude $r_{X^{-1}D}(x) \geq \kappa \;\forall x \in X_*$.
	
	Now, \eqref{eqn:A*KappaLB} follows from either of \eqref{eqn:kappa1} or \eqref{eqn:kappa2}. To prove \eqref{eqn:A*LNEnergy}, we first note that in either of the cases above we have $|X_*|\gg \kappa(\log|X|)^{-1/2}$. Then using the lower bound on $\kappa$, \eqref{eqn:t2dlb} and \eqref{eqn:tauUBLB}, we have $|X_*|^2\gg |D|\tau(\log|X|)^{-5/2}\gg E_{\times}(X)/(|X|\log|X|)^{7/2}$ as required. Finally, to deduce the required upper bound on $|D|$ in \eqref{eqn:DUBLB} note that, as shown above, $|D|\tau\ll |X_*|^2(\log|X|)^{5/2}$, which with \eqref{eqn:t2dlb} implies $|D|E_\times(X)(\log|X|)^{-1}\ll (|D|\tau)^2\ll |X_*|^4(\log|X|)^5$.
	\end{proof}

\begin{lemma}
\label{lem:EnEnlgsbipoly}
Let $X\subseteq GL_2(\F_q)$. Then there exists $X_*\subseteq X$, with
\[
|X_*| \gg \frac{E_{\times}(X)^{1/2}}{|X|^{1/2}(\log |X|)^{7/4}},
\]
such that
\begin{equation}
\label{eqn:EnEnLemLgSbvPol}
E_{+}(X_*) \ll \frac{|X_*|^4|X|^6(\log|X|)^2}{q^4 E_{\times}(X)^2} + \frac{q^{13/2}|X_*|^3|X|(\log|X|)^5}{E_{\times}(X)}.
\end{equation}
\end{lemma}
\begin{proof}
We apply Lemma~\ref{lem:EnergyEnergyPigeonh} to the set $X$ and henceforth assume its full statement, keeping the same notation. Without loss of generality, assume $r_{X^{-1}D} \geq \kappa \;\forall x\in X_*$. Thus, 
\begin{align*}
\label{eqn:EfUBE}
    E_{+}(X_*) &= |\{(x_1, x_2, x_3, x_4)\in X_*^4: x_1 + x_2 = x_3 + x_4\}|\\
    &\leq \kappa^{-2}|\{(d_1, d_2, x_1, x_2, y_1, y_2)\in D^2\times X_*^2 \times X^2: x_1 + y_1^{-1} d_1 = x_2 + y_2^{-1} d_2 \}|\\
    &= \kappa^{-2} \mathcal{I}(X^{-1}, D, -X_*, -X^{-1}, D, X_*).
\end{align*}
Then applying Proposition~\ref{thm:PPIM2} and (\ref{eqn:A*KappaLB}), we obtain
\begin{align*}
 E_{+}(X_*) &\ll \kappa^{-2}\cdot\bigg( \frac{(|D||X||X_*|)^{2}}{q^4} + q^{13/2}|D||X||X_*|\bigg)\\ 
 &\ll\frac{|X_*|^4|X|^2(\log|X|)^2}{q^4\tau^2} + \frac{q^{13/2}|X_*|^3|X|(\log|X|)^4}{|D|\tau^2}\,.
\end{align*}
Finally, applying~\eqref{eqn:tauUBLB} and \eqref{eqn:DUBLB}, we obtain the required bound in \eqref{eqn:EnEnLemLgSbvPol} for $E_{+}(X_*)$.
\end{proof}

We are now ready to give a proof of Theorem \ref{thm:BWDM2}.

\begin{proof}[Proof of Theorem~\ref{thm:BWDM2}]
We begin by describing an algorithm, which constructs two sequences of sets $A = S_1 \supseteq S_2 \cdots \supseteq S_{k+1}$ and $\emptyset = T_0 \subseteq T_1 \cdots \subseteq T_{k}$ such that $ S_{i} \sqcup T_{i-1} = A$, for $i = 1, \dots, k+1$. 

Let $1\leq M\leq |A|$ be a parameter. At any step $i\geq 1$, if $E_{\times}(S_i)\leq |A|^3/M$ the algorithm halts. Otherwise if 
\begin{equation}
\label{eqn:ESiLB}
    E_{\times}(S_i)> \frac{|A|^{3}}{M},
    \end{equation}
    through a use of Lemma \ref{lem:EnEnlgsbipoly}, with $X=S_i$, we identify a set $V_i:= X_* \subseteq S_i$, with 
\begin{equation}
\label{eqn:ModelViSize}
|V_i| \gg \frac{E_{\times}(S_i)^{1/2}}{|S_i|^{1/2}(\log|A|)^{7/4}} > \frac{|A|}{M^{1/2}(\log|A|)^{7/4}}
\end{equation}
and
\begin{equation}
    \label{eqn:AddEnBVi}
 E_{+}(V_i)\ll \frac{|V_i|^4|S_i|^6(\log|S_i|)^2}{q^4 E_{\times}(S_i)^2} + \frac{q^{13/2}|V_i|^3|S_i|(\log|S_i|)^5}{E_{\times}(S_i)}.
\end{equation}
We then set $S_{i+1} = S_i \setminus V_i$, $T_{i+1} = T_i \sqcup V_i$ and repeat this process for the step $i+1$. From \eqref{eqn:ModelViSize}, we deduce $|V_i|\gg |A|^{1/2}(\log|A|)^{-7/4}$ and so the cardinality of each $S_i$ monotonically decreases. This in turn implies that this process indeed terminates after a finite number of iterations $k$. We set $B = S_{k+1}$ and $C = T_k$, noting that $A = B\sqcup C$ and that
\begin{equation}
\label{eqn:BE+UB}
E_{\times}(B) \leq \frac{|A|^{3}}{M}.
\end{equation}

We apply the inequalities \eqref{eqn:ESiLB}, \eqref{eqn:ModelViSize} and $|S_i| \leq |A|$, to \eqref{eqn:AddEnBVi}, to get
\begin{align*}
E_{+}(V_i) &\ll M^2|V_i|^4q^{-4}(\log|A|)^2 + M|A|^{-2}|V_i|^3 q^{13/2} (\log|A|)^5
\\ &\ll \big(M^2 q^{-4}(\log|A|)^2 + M^{3/2}|A|^{-3}q^{13/2}(\log|A|)^{27/4}\big)\cdot|V_i|^4.
\end{align*}
Then, observing that
\begin{equation*}
\label{eqn:CcupVi}
 C = T_k = \bigsqcup_{i=1}^{k} V_i \subseteq A,
\end{equation*}
we use Lemma~\ref{lem:ESubadd} to obtain
\begin{align*}
    E_{+}(C) &\ll (M^2q^{-4}(\log|A|)^2 + M^{3/2}|A|^{-3}q^{13/2}(\log|A|)^{27/4}) \bigg(\sum_{i=1}^k|V_i|\bigg)^4 \\ 
    &\leq M^{2}|A|^4q^{-4}(\log|A|)^{2} + M^{3/2}|A|q^{13/2}(\log|A|)^{27/4}.
\end{align*}
Note that Lemma~\ref{lem:ESubadd} is applicable because $M_2(\F_q)$ is an Abelian group under addition.
Comparing this with \eqref{eqn:BE+UB}, we see the choice $M = M(|A|)$, given by \eqref{eqn:MAminBVPolyLS} is optimal.
\end{proof}

\section{Proofs of Theorem~\ref{thm:AddEng} and Corollary~\ref{sum-productestiamte} }
\begin{proof}[Proof of Theorem~\ref{thm:AddEng}]
We proceed similarly to the proof of \cite[Theorem~6]{RNRuShk}. Note that
\begin{align*}
    E_+(A, B) &= |C|^{-2}|\{\,(a, a^\prime, b, b^\prime, c, c^\prime)\in A^2\times B^2\times C^2: a + b c c^{-1} = a^\prime + b^\prime c^\prime (c^\prime)^{-1}\,\}|\\
    &\leq |C|^{-2}|\{\,(a, a^\prime, s, s^\prime, c, c^\prime)\in A^2\times (BC)^2\times (C^{-1})^2: a + s  c = a^\prime + s^\prime c^\prime\,\}|.
    \end{align*}
    The required result then follows by applying Proposition~\ref{thm:PPIM2}.
\end{proof}

\begin{proof}[Proof of Corollary~\ref{sum-productestiamte}]
Since $|A|\gg q^3$, we may assume $A\subseteq GL_2(\F_q)$. We use Theorem~\ref{thm:AddEng}, with $A=B=C$ and apply the lower bound on $E_+(A)$ given by \eqref{eqn:AddECS} to obtain \eqref{eqn:SRB1}. To prove \eqref{eqq99}, we follow the same process and apply the assumption $|AA|\ll |A|$, to obtain $|A+A|\gg \min\{\,q^4,\, |A|^3/q^{13/2}\,\}$, which gives the required result.
\end{proof}

\section{Proofs of Theorem~\ref{thm:3vexp} and Theorem \ref{thm:SarkM2v2}}
\begin{proof}[Proof of Theorem~\ref{thm:3vexp}]
For $\lambda \in AB +C$, write 
\[
t(\lambda) = |\{\,(a, b, c)\in A\times B \times C: a b + c = \lambda\,\}|.
\]
By the Cauchy-Schwarz inequality, we have
\[
(|A||B||C|)^2 = \left(\sum_{\lambda\in AB+C} t(\lambda)\right)^2 \leq |AB +C|\sum_{\lambda\in AB+C} t(\lambda)^2.
\]
Further noting that
\[
\sum_{\lambda\in AB+C} t(\lambda)^2 = \mathcal{I}(A, B, -C, -A, B, C).
\]
We apply Proposition~\ref{thm:PPIM2} to obtain
\[
|AB + C|\gg \min\left\{\,q^4,\, \frac{|A||B||C|}{q^{13/2}}\,\right\}.
\]
This immediately implies the required result.

For the set $(A+B)C$, as above we have 
\[|(A+B)C|\ge \frac{|A|^2|B|^2|C|^2}{|\{\,(a, b, c, a', b', c')\in (A\times B\times C)^2\colon (a+b)c=(a'+b')c'\,\}|}.\]
To estimate the denominator, we follow the argument in the proof of Proposition \ref{thm:PPIM2}. In particular, we first define a graph $G$ with the vertex set $V=M_2(\mathbb{F}_q)\times M_2(\mathbb{F}_q)\times M_2(\mathbb{F}_q)$, and there is a direct edge going from $(a, e, c)$ to $(b, f, d)$ if $b a+e f=c+d$. The only difference here compared to that graph in Section $3$ is that we switch between $b a$ and $a b$. By using a similar argument as in Section $3$, we have this graph is a $(q^{12}, q^8,c q^{13/2})$-digraph, where $c$ is a positive constant. 

To bound the denominator, we observe that the equation 
\begin{equation*}
	(a+b)c=(a'+b')c'
\end{equation*}
gives us a direct edge from $(c, -b', -ac )$ to $(b, c', a'c')$. So let $U:=\{(c, -b', -ac)\colon a\in A, c\in C, b'\in B\}$ and $W=\{(b, c', a'c')\colon b\in B, c'\in C, a'\in A\}$. Since $C\subseteq GL_2(\mathbb{F}_q)$, we have $|U|=|W|=|A||B||C|$. So applying Lemma \ref{edge}, the number of edges from $U$ to $W$ is at most 
\[\frac{|A|^2|B|^2|C|^2}{q^4}+q^{13/2}|A||B||C|.\]
In other words, 
\[|\{\,(a, b, c, a', b', c')\in (A\times B\times C)^2\colon (a+b)c=(a'+b')c'\,\}|\ll \frac{|A|^2|B|^2|C|^2}{q^4}+q^{13/2}|A||B||C|,\]
and we get the desired estimate.
\end{proof}
\begin{proof}[Proof of Theorem \ref{thm:SarkM2v2}]
By the Cauchy-Schwarz inequality and Proposition~\ref{thm:PPIM2}, we have
\begin{align*}
   \mathcal{J} &= |\{\,(a, b, c, d)\in A\times B \times C\times D: a+b = c d\,\}|\\
    &\leq |B|^{1/2}|\{\,(a, a^\prime, c, c^\prime, d, d^\prime) \in A^2 \times C^2 \times D^2: c d - a = c^\prime d^\prime - a^\prime\,\}|^{1/2}\\
    &\ll  \frac{|A||B|^{1/2}|C||D|}{q^2} + q^{13/4} (|A||B||C||D|)^{1/2}.
\end{align*}
\end{proof}

\section*{Acknowledgments}
Thang Pham was supported by Swiss National Science Foundation grant P4P4P2-191067.


\begin{thebibliography}{}

\bibitem{BW} A. Balog and T.D. Wooley, {\it A low–energy decomposition theorem}, {\it Quart. J. Math.\/}, {\bf 68} (2017), 207--226.

\bibitem{HH}
N. Hegyv\'{a}ri and F. Hennecart, \textit{Expansion for cubes in the Heisenberg group,} {\it Forum Math.}, \textbf{30} (2018), 227--236.

\bibitem{KaKoPhShVi} Y.D. Karabulut, D. Koh, T. Pham, C-Y. Shen and L. A. Vinh, {\it Expanding phenomena over matrix rings}, {\it Forum Math.\/}, {\bf 31} (2019), 951--970.

\bibitem{MohSte} A. Mohammadi and S. Stevens, {\it Low-energy decomposition results over finite fields}, preprint {\tt arXiv:2102.01655v1 [math.CO]} (2021).
\bibitem{MS54}
A. Mohammadi and S. Stevens, \textit{Attaining the exponent $5/4$ for the sum-product problem in finite fields}, preprint {\tt arXiv:2103.08252 [math.CO]} (2021).

\bibitem{MurPetri18} B. Murphy and G. Petridis, {\it Products of differences over arbitrary finite fields}, {\it Discrete Analysis}, {\bf 18} (2018), 1--42.

\bibitem{pham}
T. Pham and L. A. Vinh, \textit{Distribution of distances in vector spaces over prime fields}, {\it Pacific Journal of Mathematics}, \textbf{309} (2020), 437--451.
\bibitem{RNRuShk} O. Roche-Newton, M. Rudnev and I. D. Shkredov,
{\it New sum-product type estimates over finite fields},
{\it Adv. Math.} {\bf 293} (2016), 589--605.

\bibitem{RNShWin} O. Roche-Newton, I. E. Shparlinski and A. Winterhof, {\it Analogues of the Balog-Wooley decomposition for subsets of finite fields and character sums with convolutions}, {\it Ann. Comb.\/}, {\bf 23} (2019), 183--205.

\bibitem{RudShSt} M. Rudnev, I. D. Shkredov and S. Stevens, {\it On the energy variant of the sum-product conjecture}, {\it Rev. Mat. Iberoam\/}, {\bf 36(1)} (2020), 207--232.

\bibitem{Sar} A. S\'ark\"ozy, {\it On sums and products of residues modulo $p$}, {\it Acta Arith.\/}, {\bf 118} (2005), 403--409.


\bibitem{Shk16} I. D. Shkredov, {\it An application of the sum-product phenomenon to sets having no solutions of several linear equations}, preprint {\tt arXiv:1609.06489 [math.NT]} (2016).

\bibitem{Shk20} I. D. Shkredov, {\it A remark on sets with small Wiener norm}, in: {\it Trigonometric Sums and Their Applications.\/}, Springer, Cham. (2020).

\bibitem{SwaWin} C. Swaenepoel and A. Winterhof, {\it Additive double character sums over structured sets and applications}, {\it Acta Arith.\/}, to appear.

\bibitem{VA}
D. N. Van Anh, L. Q. Ham, D. Koh, T. Pham and L. A. Vinh, \textit{On a theorem of Hegyv\'{a}ri and Hennecart}, {\it Pacific Journal of Mathematics}, \textbf{305}(2) (2020), 407--421.
\bibitem{vu}
V. Vu, {\it Sum-product estimates via directed expanders}, {\it Math. Res. Lett.}, {\bf 15} (2008), 375--388.

\end{thebibliography}
\end{document}